\documentclass[12pt]{amsart}

\usepackage{amssymb,amsbsy,amsmath,amsfonts,amssymb,amscd}

\input xy
\xyoption{all}

\usepackage{tikz}
\usetikzlibrary{decorations.markings}
\usetikzlibrary{arrows}
\usepgflibrary{shapes}
\usepackage{latexsym}
\usepackage{graphics}
\usepackage{amsmath,latexsym,amssymb}
\usepackage[german,english]{babel}
\usepackage{color}

\input xy
\xyoption{all}

\newcommand\sC{{\mathcal C}}

\newcommand\sB{{\mathcal B}}

\newcommand\Ga{\Gamma}

\newcommand\ga{\gamma}

\def\Bbb{\bf}
\newcommand{\CC}{\ensuremath{\mathbb{C}}}

\newcommand{\ZZ}{\ensuremath{\mathbb{Z}}}
\newcommand{\QQ}{\ensuremath{\mathbb{Q}}}

\newcommand{\NN}{\ensuremath{\mathbb{N}}}

\newcommand{\HH}{\ensuremath{\mathbb{H}}}

\newcommand{\ra}{\ensuremath{\rightarrow}}

\def\eea{\end{eqnarray*}}
\def\bea{\begin{eqnarray*}}

\newcommand\dual{\mathrel{\raise3pt\hbox{$\underline{\mathrm{\thinspace d
\thinspace}}$}}}

\newcommand\qe{\ifhmode\unskip\nobreak\fi\quad $\Box$}

\def\BOX{\hfill\lower.5\baselineskip\hbox{$\Box$}}

\newtheorem{theo}{Theorem}[section]
\newtheorem{remarkk}[theo]{Remark}
\newenvironment{rem}{\begin{remarkk}\rm}{\end{remarkk}}

\newtheorem{defin}[theo]{Definition}
\newenvironment{definition}{\begin{defin}\rm}{\end{defin}}

\newtheorem{prop}[theo] {Proposition}

\newtheorem{lemma}[theo]{Lemma}
\newtheorem{example}[theo]{Example}

\renewcommand{\a}{\alpha}
\renewcommand{\b}{\beta}

\newcommand{\e}{\varepsilon}

\newcommand{\z}{\zeta}

\newcommand{\x}{\xi}
\newcommand{\s}{\sigma}

\newcommand{\inv}{^{^{-1}}}

\newcommand{\cutoff}[1]{}

\DeclareMathOperator{\Aut}{Aut}

\begin{document}

\title[Genus stabilization for moduli of curves with  symmetries]
{Genus stabilization for the components of moduli spaces 
of curves with symmetries}
\author{Fabrizio Catanese, Michael L\"onne, Fabio Perroni}

\thanks{The present work took place in the realm of the DFG
Forschergruppe 790 "Classification of algebraic
surfaces and compact complex manifolds".
}

\date{\today}

\maketitle

\begin{abstract}

In a previous paper, \cite{CLP12}, we introduced a new homological invariant  $\e$ for the faithful action of a
 finite group $G$ on an algebraic curve.

We show here  that the moduli space of curves admitting a faithful 
action of a finite group $G$  with a fixed  homological invariant $\e$, if the genus $g'$ of  the quotient curve satisfies 
 $g'>>0$, is irreducible (and non empty iff the class satisfies the `admissibility' condition).
 
We achieve this by showing that   the stable classes are in 
bijection with the set of admissible classes  $\e$.

\end{abstract}


\section{Introduction}

The main purpose of this article is the determination of the irreducible components of 
the moduli spaces of curves admitting a given  symmetry group $G$  and  to set up the stage
for the investigation of their  homological stabilization.

Let us first  introduce  the framework into  which the stabilization theorem proven in this paper
 fits.

Let $g>1$ be a positive integer, and let $C$ be a projective curve of genus $g$.  Consider a subgroup $G$ of the group of automorphisms
of $C$, and let   $\ga \in G \leq {\rm Aut}(C)$: then, 
since $C$ is a $K (\Pi_g, 1)$ space, the homotopy class of
$
\ga \colon C \to C \, 
$
 is determined by 
$$
\pi_1(\ga) \colon \pi_1(C,y_0) \to \pi_1(C, \ga \cdot y_0) \, 
$$
and actually  the topological type of the action is completely 
determined by its action on the fundamental group, up to inner 
automorphisms.
Hence we get a group homomorphism
$$
\rho \colon G \to \frac{{\rm Aut}(\Pi_g)}{{\rm Inn}(\Pi_g)}={\rm Out}(\Pi_g) \, .
$$

This homomorphism $\rho$ is injective by  Lefschetz'  lemma 
$
( \ga \underset{h}{\sim} id_C \Rightarrow \ga =id_C \, )
$
asserting that, for $ g \geq 2$, only the identity transformation  is homotopic to the identity.
Moreover, since holomorphic maps are orientation preserving, we have actually   
$$\rho \colon G \to {\rm Out}^+(\Pi_g)={\rm Map}_g \cong \frac{{\rm Diff}^+(C)}{{\rm Diff}^0(C)},$$
hence indeed $\rho$ determines the differentiable type of the action.

When we speak of the moduli space of curves admitting an action by a finite group $G$ of a fixed  topological type 
we fix  $\rho(G)$ up to  the action of  ${\rm Aut}(G)$ on the source and the adjoint action of  
the mapping class group ${\rm Map}_g$
on the target.
In order to be more precise, let us recall the description of the 
moduli spaces of curves via Teichm\"uller theory.

Fix $M$ to be the underlying oriented differentiable manifold of $C$, 
and let $\mathcal{C}S(M)$ the space of complex structures on $M$:
then  the moduli space of curves
$\mathfrak{M}_g $, and Teichm\"uller space $\mathcal{T}_g $ are defined as:

$$
\mathfrak{M}_g := \mathcal{C}S(M)/{\rm Diff}^+(M), \ 
\mathcal{T}_g := \mathcal{C}S(M)/{\rm Diff}^0(M), \ 
\Rightarrow \mathfrak{M}_g = \mathcal{T}_g/ {\rm Map}_g.
$$

Teichm\"uller's theorem says that $\mathcal{T}_g \subset \CC^{3g-3}$ is an open subset homeomorphic to a ball. 
Moreover ${\rm Map}_g$ acts properly discontinuously on $\mathcal{T}_g$, but not freely.

As a corollary, the rational cohomology of the moduli space is calculated by group cohomology:
$
H^*(\mathfrak{M}_g, \QQ) \cong H^*({\rm Map}_g, \QQ) \, .
$

Harer (\cite{harer}) showed that these cohomology groups stabilize with $g$,
while the ring structure of the cohomology of the ``stable mapping class group'' is described by a conjecture of Mumford,
 proven by Madsen and Weiss  (\cite{m-w}).

\begin{defin}
$
\mathcal{T}_{g,\rho} := \mathcal{T}_g^{\rho(G)}
$
is the fixed-point locus of $\rho(G)$ in $\mathcal{T}_g$.
\end{defin}

The first author (\cite{FabIso})   proved an analogue of Teichm\"uller's theorem,   that
$\mathcal{T}_{g,\rho}$ is homeomorphic to a ball. Hence we get other
subvarieties of the moduli space 
according to the following definition:

\begin{defin}
$\mathfrak M_{g,\rho}$ is the (irreducible closed subvariety)  image of $\mathcal{T}_{g,\rho}$ in $\mathfrak{M}_g$.
That is, $\mathfrak M_{g,\rho}$  is the quotient of $\mathcal{T}_{g,\rho}$ by the normalizer
of $\rho(G)\leq {\rm Map}_g$.

Instead the marked moduli space $\mathfrak M\mathfrak M_{g,\rho}$ is the quotient of $\mathcal{T}_{g,\rho}$ by the centralizer
of $\rho(G)\leq {\rm Map}_g$ (here the action of $G$ is given, together with a marking of $G$, and we do not allow 
 to change the given action of $G$ up to an automorphism
in $Aut (G)$.

\end{defin}

The above description of the space of curves admitting a given topological action of a group $G$ is quite nice,
but not completely explicit. One can make everything more explicit via a more precise 
understanding through geometry: in this way one can find discrete invariants of the topological type of the action.

Consider the quotient curve  $C':=C/G$, let $g':=g(C')$ be  its genus ,  let $\sB=\{ y_1, \ldots , y_d\}$ be the branch locus, and 
let $m_i$  be the multiplicity
of $y_i$. \\

$C\to C'$ is determined (by virtue of Riemann's existence theorem)  by the monodromy:
$
\mu \colon \pi_1(C'\setminus \sB, y_0) \to G \, .
$
We have:

\begin{small}
$$ \pi_1(C' \setminus \sB, y_0) \cong \Pi_{g', d}:= \langle \ga_1, \ldots , \ga_d, \a_1, \b_1 , \ldots , \a_{g'}, \b_{g'} | \prod_{i=1}^d \ga_i \prod_{j=1}^{g'} [\a_j, \b_j]=1 \rangle .$$
\end{small}

\noindent
Therefore the datum of $\mu $ is equivalent to the datum of a Hurwitz generating system, i.e. of a vector 
$$
v : = (c_1, \ldots , c_d, a_1, b_1, \ldots , a_{g'}, b_{g'}) \in G^{d+2g'}
$$
s.t. 

\begin{enumerate}
\item
$G$ is generated by the entries $c_1, \ldots , c_d, a_1, b_1, \ldots , a_{g'}, b_{g'}$,
\item
$c_i \not= 1_G$, $\forall i$, and 
\item
$$\prod_{i=1}^d c_i \prod_{j=1}^{g'}[a_j,b_j]=1.$$
\end{enumerate}

Riemann's existence theorem shows that  the irreducible 
components of the moduli space corresponding  to
curves with a given group $G$ of automorphisms 
(i.e., $G$ acts faithfully, or, equivalently, 
the monodromy $\mu$ is surjective!) 
are determined by the possible topological types 
of the branched covering $ C \ra C' = C / G$, which
are in bijection with the points of the
orbit space
$$
{\rm Epi}(\Pi_{g',d}, G)/ ({\rm Map}_{g',d}\times {\rm Aut}(G)) \, 
$$
(here `Epi' stands for `Set of Epimorphisms').
 \smallskip
 
\noindent
{\bf Relation of the two approaches through the orbifold fundamental group $\pi_1^{\rm orb}$.}

Let $X$ be a  ``good'' topological space and let $\tilde{X}$ be the universal cover of $X$.
Assume that $G$ is a finite group acting effectively  on $X$.
Then $G$ can be lifted to a discontinuous group $\tilde{G}$
of homeomorphisms of $\tilde{X}$ in such a way that the quotients $\tilde{X}/\tilde{G}$
and $X/G$ are homeomorphic. $\tilde{G}$ is called the 
{ \bf orbifold fundamental group} of the $G$-action and it is an extension of $\pi_1(X)$
by $G$, i.e., we have a short exact sequence
\begin{small}
$$
1 \to \pi_1(X) \to \tilde{G} = :  \pi_1^{\rm orb}(X/G)  \to G \to 1 \, .
$$
\end{small}
In our situation, $X=C$, and its universal cover is the upper half plane $\tilde{X}=\HH$, hence $ \pi_1^{\rm orb}(X/G)$
is a Fuchsian group. 

We have:
$$
C=\HH/\Pi_g\, , \qquad C'=\HH/\pi_1^{\rm orb}, \, \qquad
\pi_1^{\rm orb} =\Pi_{g',d}/\langle \langle \ga_1^{m_1}, \ldots , \ga_{d}^{m_d} \rangle \rangle \, .
$$
where, as usual, $\langle \langle a_1, \ldots , a_{d} \rangle \rangle $
denotes the subgroup normally generated by the elements $a_1, \ldots , a_{d}$.

The above exact sequence yields, via conjugation acting on the normal subgroup  
$ \pi_1(X) \cong \Pi_g$, a homomorphism  $\rho \colon G \ra Out (\Pi_g)$.

Fix now $d$ points $y_1, \dots , y_d$ on the oriented differentiable
manifold $M'$ underlying $C'$, and set $\sB : = \{  y_1, \dots , y_d\}$.
One can then define  the d-marked Teichm\"uller space as the quotient 
$$
\mathcal{T}_{g',d} := \mathcal{C}S(M')/{\rm Diff}^0(M', \sB). 
$$
We have thus a covering map
of connected spaces
$\mathcal{T}_{g',d} \to \mathcal{T}_{g,\rho} \, ;$
since $\mathcal{T}_{g,\rho}$ is a ball  the above map is a homeomorphism 
(hence  in particular also 
 the topological type $\rho$ determines the monodromy $\mu$, as can  be proven directly).

We can now describe the numerical  and the  homological invariants of $\rho$ (equivalently, of $\mu$ or of  $v$).

The first numerical invariant is called (by different authors)  the {\bf branching numerical function}/Nielsen class/$\nu$-type.

\begin{defin}
Let ${\rm Conj}(G) $ be the set of (nontrivial) conjugacy classes in the group $G$: then we  let 
 $\nu \colon {\rm Conj}(G) \to \NN$, $\nu (\mathcal{C}):= | \{ i | c_i \in \mathcal{C}\}|$, be the function
 which counts how many  local monodromy elements (these are defined only up to conjugation)
 are in a fixed conjugacy class.
 
 \end{defin}
 
This invariant was first introduced  by Nielsen (\cite{nielsen}) who proved that 
$\nu$ determines $\rho$ if $G$ is cyclic.

\bigskip

The (semi-)classical homological invariant is instead defined as follows.

Let $H:=\langle \langle c_1, \ldots , c_d \rangle \rangle$ be the subgroup normally generated by the local monodromies
(local generators of the isotropy subgroups). Consider  the quotient group $G'' := G/H$.

Then the covering $C\to C'=C/G$ factors through $C''=C/H \to C'$, which is unramified and with group $G''$.
The monodromy $\mu'' \colon \Pi_{g'} \to G''$ corresponds to a homotopy class of a continuous map 
$ m'' : C' \to {\rm K}(G'', 1)$. Passing to homology we have
$$
H_2 (m'') = H_2(\mu'') \colon H_{2}(\Pi_{g'} , \ZZ) = H_{2}( C' , \ZZ) \to H_2(G'', \ZZ) \, .
$$
One defines then $h(v)$ as the image of  the fundamental class of $C'$, the generator $ [C']$ of $ H_{2}( C' , \ZZ)$
determined by the complex orientation:
$$
h(v):=H_2(\mu'')[C'] \in H_2(G'', \ZZ) \, .
$$

More generally Edmonds (\cite{Edm I, Edm II}) showed that 
$\nu$ and $h$ determine $\rho$ for $G$ abelian, and that, 
 if moreover $G$ is split-metacyclic and the action is free (i.e., $G = G''$), then $h$
determines $\rho$.

\medskip

In our recent paper \cite{CLP12} we 
considered the case  $G=D_n$
of  the dihedral group of order $2n$. After  showing that in this case  $(\nu,h)$ does not determine $\rho$,
we introduced therefore  a  finer homological invariant $\e \in G_{ \Ga}$, where $\Ga$ is the union
of the conjugacy classes of the local monodromies $c_i$, and $ G_{ \Ga}$ 
is a group constructed from the given group $G$ and  the given subset $\Ga$.
The main result of  \cite{CLP12} was to  prove that in the case of the dihedral  group our invariant   $\e$ determines the class of $\rho$.

We do not recall here the definition of the group $ G_{ \Ga}$ and of the invariant $\e (v)$,
since the whole section two is devoted to these definitions and their topological interpretation
(which was not contained in \cite{CLP12}).
It suffices here to observe that this new invariant encodes in particular all the classical numerical and homological invariants. 

It would be a too nice and simple world if this invariant would do the job for any group. But this invariant,
which in the unramified case coincides with the classical invariant in $H_2 (G, \ZZ)$, does not distinguish irreducible components in general (see \cite{DT}).

It  only does so if the genus $g'$ of the quotient curve is large enough. The crucial point, as in Harer's theorem,
is the concept of stabilization.

\noindent \underline{Direct stabilization:} $\mu$ stabilizes by extending 
$
\mu(\a_{g'+1})= \mu(\b_{g'+1})=1 \in G \, .
$
In other words, we add a handle to the quotient curve $C'$, such that on it the monodromy is trivial.

\noindent \underline{Stabilization:} is defined as
the equivalence relation generated by direct stabilization.

In the \'etale case it was shown that the homology invariant is a full  `stable' invariant.

\begin{theo}[Dunfield-Thurston]
In the unramified case ($d=0$), for $g' >> 0$, the  equivalence classes are in bijection with 
$$
\frac{H_2(G,\ZZ)}{{\rm Aut}(G)} \, .
$$
\end{theo}

The proof of the above theorem is based on the interpretation of second homology as bordism,
and on Livingston 's theorem (\cite{Livingston}) showing that two unramified monodromies having the 
same homology class in $H_2(G,\ZZ)$ are stably equivalent.  A very suggestive proof of Livingston 's theorem,
based on the
concept of a relative Morse function with increasing  Morse indices, is given in \cite{DT}, while
an algebraic proof was given by Zimmermann in \cite{Zimmer}.

In the ramified case, the situation is much more complicated, and it turns out to be  safer
to rely on the algebraic technique of Zimmermann in order to set up a secure, even if technical,   proof of
 the following main theorem:
 
\begin{theo}\label{stablebranched}
For $g'>>0$, the equivalence classes are in bijection with the set of admissible classes  $\e$.
\end{theo}

In the above theorem, the condition of admissibility is the simple translation of  the condition that the 
product of the local monodromies $c_1 \dots c_d$ must be a product of commutators.

In a sequel to this paper we shall   prove another stabilization, which we call branching stabilization, and 
which generalizes the following result (see \cite{FV})
\begin{theo}[Conway -Parker]
In the  case $g'=0$, let $G = F / R$ where $F$ is a free group, and  assume that $  H_2(G,\ZZ) \cong \frac{[F,F]\cap R}{[F,R]}$ is generated by commutators. Then  there is an integer $N$ such that if the numerical function $\nu$ takes values $\geq N$,

then there is only one   equivalence class  with the given  numerical function $\nu$.
\end{theo}

We shall  get  an analogous result for any genus $g'$, using our fuller homological invariant $\e$.  In the course of
proving branching stabilization, we shall also give a different proof of our genus stabilization result, using
a variant of the semi-group and of the group introduced by Conway and Parker.

Finally, we want to mention the  interesting question of determining the class of groups for which no stabilization is needed, 
thus extending to other groups the result  we obtained for
the dihedral groups.

\section{The $\e$-invariant}
In this Section we first review the definition of the $\e$-invariant of Hurwitz vectors which was introduced in \cite[Sec. 3]{CLP12}, 
then we give a topological interpretation of this invariant as a class in a certain relative homology group of $G$
modulo an equivalence relation.   Although we don't use explicitly this topological interpretation in the proof of the main theorem,
many of the technical results which we prove throughout the paper are just algebraic reformulations of simple geometrical
statements whose understanding is easier using a topological view of the $\e$-invariant.   

Let us first recall the following
\begin{definition}\label{factorisations}
Let $G$ be a finite group, and let $g',d \in \NN$. A $g',d$-\emph{Hurwitz vector
in} $G$ is an element $v\in G^{d+2g'}$, the Cartesian product of $G$ $(d+2g')$-times. 
A $g',d$-Hurwitz vector in $G$ will  also be denoted by
$$
v=(c_1, \dots , c_d \, ; \, a_1, b_1, \dots , a_{g'},b_{g'})\, .
$$
For any $i\in \{ 1, \dots , d+2g'\}$, the $i$-th component $v_i$ of $v$ is defined as usual.
The \emph{evaluation} of $v$ is the element 
$$
ev(v): =\prod_1^d c_j \cdot \prod_1^{g'}[a_i, b_i] \in G \, .
$$ 

A \emph{Hurwitz generating system of length $d+2g'$ in $G$} is a $g',d$-Hurwitz vector $v$ in $G$
such that the following conditions hold:
\begin{itemize}
\item[(i)] $c_i\not=1$ for all $i$;
\item[(ii)] $G$ is generated by the components  $v_i$ of 
$v$;
\item[(iii)] $\prod_1^d c_j \cdot \prod_{1}^{g'}[ a_i , b_i] =1$.
\end{itemize}
We denote by $HS(G;g',d)\subset G^{2g'+d}$ the set of all Hurwitz generating systems in $G$ of
type $(g', d)$ (hence of  length $d+2g'$).
\end{definition}

\begin{definition}\label{GGamma}
Let $G$ be a finite group.

>From now on, $F: =\langle \hat{g}\, | \, g\in G\rangle$
shall be  the free group generated by the  elements of $G$. 
Let $R\unlhd F$ be the normal subgroup of relations, so  that 
$G=\frac{F}{R}$.

For any union of non-trivial conjugacy classes $\Ga \subset G$,  define
\begin{align*}
& R_\Ga : = \langle \langle [F,R], \hat{a}\hat{b}\hat{c}^{-1}\hat{b}^{-1}\, | \, \forall a, c  \in \Ga ,b \in G \ s.t. \  ab=bc\rangle \rangle \, , \\
& G_\Ga: = \frac{F}{R_\Ga} \, .
\end{align*}
The map $\hat{a}\mapsto a$, $\forall a \in G$, induces a group homomorphism $\a \colon G_\Ga \to G$ whose kernel 
shall be denoted by  $K_\Ga : =Ker(\a)$.
\end{definition}

By \cite[Lemma 3.2]{CLP12}, $K_\Ga = \frac{R}{R_\Ga}$ is contained in the centre of $G_\Ga$.

Finally we set:
\begin{definition}\label{ev}
Given a $g',d$-Hurwitz vector 
$$
v=(c_1, \dots , c_d; a_1, b_1 , \dots , a_{g'}, b_{g'})
$$
 in $G$ (cf. Definition \ref{factorisations}), 
 its \emph{tautological lift}, $\hat{v}$, is the $g', d$-Hurwitz vector in $G_\Ga$ defined by
 $$
 \hat{v}: = (\widehat{c_1}, \dots , \widehat{c_d}; \widehat{a_1}, \widehat{b_1}, \dots ,  \widehat{a_{g'}}, \widehat{b_{g'}})
 $$ 
 whose components  are the tautological lifts of the components of $v$. 
 
Given a $g',d$-Hurwitz vector $v$ in $G$ with $c_i\not= 1$, $\forall i= 1, \dots d$, we denote by $\Ga_v$ the union of all 
the conjugacy classes of $G$
containing at least one $c_i$.

For any Hurwitz generating system $v$  (or Hurwitz vector satisfying (iii)), set
$$
\e(v): = \prod_1^d \widehat{c_j} \cdot \prod_1^{g'}[\widehat{a_i}, \widehat{b_i}]\in K_{\Ga_v} \, ,
$$ 
to be  the evaluation of the tautological lift $\hat{v}$ of $v$ in $G_{\Ga_v}$ (cf. Definition \ref{factorisations}).
\end{definition}

By \cite[Lemma 3.5]{CLP12}, any automorphism $f\in \Aut (G)$ induces an isomorphism $f_\Ga \colon K_\Ga \to K_{f(\Ga)}$ in a natural way,
such that $\e(f(v))=f_\Ga (\e(v))$, where $\Ga=\Ga_v$. Therefore the map
$$
\e \colon HS(G; g',d) \to \coprod_\Ga K_\Ga \, , \quad v \mapsto \e(v) \in K_{\Ga_v} 
$$
where $\coprod_\Ga K_\Ga$ denotes the disjoint union of the $\Ga$'s, 
descends to a map
$$
\tilde{\e} \colon HS(G;g' , d)/_{\Aut (G)} \to \left(  \coprod_\Ga K_\Ga \right)/_{\Aut (G)} \, .
$$
The key point here is that $\tilde{\e}$ is invariant under the action of the mapping class group $Map(g',d)$ \cite[Prop. 3.6]{CLP12},
hence giving a well defined map
\begin{equation}\label{ehat}
\hat{\e} \colon \left( HS(G;g' , d)/_{\Aut (G)} \right)/_{Map(g',d)} 
\to \left(  \coprod_\Ga K_\Ga \right)/_{\Aut (G)} \, .
\end{equation}
One of the main results of  \cite{CLP12} says that, when $G=D_n$, $\hat{\e}$ is injective ({\it loc. cit.} Thm. 5.1).
With the aid of this we reached   a classification of  the orbits of Hurwitz generating systems for $D_n$
under the action of $Map(g',d)$ modulo automorphisms in $\Aut (G)$.

\medskip

For later use, let us recall the following definition \cite[Def. 3.11]{CLP12}.
\begin{definition}\label{H2Gamma}
Let $\Ga \subset G$ be a union of non-trivial conjugacy classes of $G$. We define
$$
H_{2,\Ga}(G)=\ker \left( G_\Ga \to G\times G_\Ga^{ab} \right) \, ,
$$
where $G_\Ga \to G\times G_\Ga^{ab}$ is the morphism with first component $\a$ (defined in Definition \ref{GGamma})
and second component the natural epimorphism $G_\Ga \to  G_\Ga^{ab}$.
\end{definition}
Notice that 
$$
H_2(G,\ZZ)\cong \frac{R\cap [F,F]}{[F,R]} \cong \ker \left( \frac{F}{[F,R]} \to G\times G^{ab}_\emptyset \right) \, .
$$
In particular, when $\Ga = \emptyset$, $H_{2,\Ga}(G)\cong H_2(G, \ZZ)$.

By \cite[Lemma 3.12]{CLP12} we have that the morphism
$$
R\cap [F,F] \to \frac{R}{R_\Ga} \, , \quad r \mapsto rR_\Ga
$$
induces a surjective group homomorphism 
\begin{equation}\label{h2toh2ga}
H_2(G,\ZZ) \to H_{2, \Ga}(G) \, .
\end{equation}


\subsection{A topological interpretation  of $\e (v)$}

For a finite group $G$, let $BG$ be the CW-complex defined as follows (see \cite[Ch. V, 7.]{Whitehead} for more details).
The $0$-skeleton $BG^0$ consists of one point. The $1$-skeleton 
$$
BG^1= \bigvee_{g\in G}S^1_{g}
$$
is a wedge  of circles indexed by $g\in G$ meeting in $BG^0$. Then the fundamental group has a canonical isomorphism $\pi_1(BG^1)\cong F$, the isomorphism being given by
sending a generator of $\pi_1(S^1_g)$ to $\hat{g}$. In this way we get an epimorphism $\pi_1(BG^1)\to G$ whose kernel
is identified with $R$ by the previous isomorphism. For any $r\in R$, let $h_r \colon S^1 \to BG^1$
be the  continuous map such that the image of a chosen generator of $\pi_1(S^1)$ under $(h_r)_*$ is $r\in \pi_1(BG^1)$. 
Using $h_r$ we attach the $2$-cell $E^2_r$ to $BG^1$. This gives  the $2$-skeleton $BG^2$,
with the property that $\pi_1(BG^2)\cong G$. The $3$-skeleton $BG^3$ is defined by attaching  $3$-cells to $BG^2$ in such a way that 
$\pi_2(BG^3)=0$ and, by induction, the $(n+1)$-skeleton $BG^{n+1}$ is defined similarly in order to get $\pi_n(BG^{n+1})=0$, $n\geq 2$.
The CW-complex $BG$ is the inductive limit $BG : =  \cup \{ BG^n\, | \, n\geq 0\}$ thus obtained. 
By construction $BG$ is an Eilenberg-Mac Lane space of type $(G,1)$, i.e. a $K(G,1)$-space.
Furthermore there is a principal $G$-bundle $EG \to BG$ with $EG$ contractible.

\begin{lemma}\label{relhopf}
Let $G$ be a finite group and let $BG$ be the CW-complex defined above. There is an isomorphism 
$$
\frac{R}{[F,R]} \cong H_2(BG,BG^1)
$$
such that the following diagram is commutative:
$$
\xymatrix{
0 \ar[r] & H_2(BG) \ar[d] \ar[r] &H_2(BG,BG^1) \ar[d] \ar[r] & H_1(BG^1) \ar[d] \ar[r] & H_1(BG) \ar[d] \ar[r] & 0 \\
0 \ar[r] & \frac{R\cap[F,F]}{[F,R]} \ar[r] & \frac{R}{[F,R]}  \ar[r] & F^{ab} \ar[r] & G^{ab} \ar[r] & 0
}
$$
where all the homology groups are with integer coefficients,
 the horizontal sequences are exact  (the upper one being part of the homology sequence of the pair $(BG,BG^1)$) 
 and the vertical arrows are isomorphisms (the one on the left is given by Hopf's theorem, see \cite{Hopf}, the two on the right 
 are the canonical isomorphisms $H_1 \cong \pi_1^{ab}$).
\end{lemma}
\begin{proof}
Recall that the homology of a CW-complex $K=\{ K^n \}_{n\in \NN}$ can be computed as follows (cf. \cite[Ch. IX]{Massey}).
Define 
\begin{eqnarray*}
&&C_n(K):=H_n(K^n,K^{n-1}) \, , \\
&&\partial_n \colon C_n(K) \to C_{n-1}(K)\, , \quad \partial_n =j_{n-1}\circ\partial_* \, ,
\end{eqnarray*}
where $\partial_* \colon H_n(K^n,K^{n-1})\to H_{n-1}(K^{n-1})$ is the boundary operator of the long exact sequence 
of the pair $(K^n,K^{n-1})$ and $j_{n-1} \colon H_{n-1}(K^{n-1}) \to H_{n-1}(K^{n-1}, K^{n-2})$ is the homomorphism
induced by the inclusion map. Then $\{C_\bullet (K), \partial_\bullet\}$ is a complex and
 $H_n(K)\cong H_n(C_\bullet (K))$.

We regard $BG^1$ as a $1$-dimensional CW-complex, so the above construction gives the complexes $C_\bullet(BG)$
and $C_\bullet(BG^1)$. Notice that the inclusion $BG^1 \to BG$ gives an injective map 
$C_\bullet(BG^1)\to C_\bullet (BG)$ and define 
$$
C_\bullet (BG,BG^1):=\frac{C_\bullet(BG)}{C_\bullet(BG^1)} \, .
$$
Since $C_n(K)$ is the free group with basis in $1$-$1$
correspondence with the $n$-cells of $K$, we obtain the diagram
\begin{equation}\label{diag}
\xymatrix{
& \vdots \ar[d] & \vdots \ar[d] & \vdots \ar[d] & \\
0 \ar[r] & 0 \ar[d] \ar[r] & C_3(BG)\ar[d] \ar[r] & C_3(BG,BG^1) \ar[d] \ar[r] & 0 \\
0 \ar[r] & 0 \ar[d] \ar[r] & C_2(BG)\ar[d] \ar[r] & C_2(BG,BG^1) \ar[d] \ar[r] & 0 \\
0 \ar[r] & C_1(BG^1)\ar[d] \ar[r] & C_1(BG)\ar[d] \ar[r] & 0 \ar[d] \ar[r] & 0 \\
0 \ar[r] & C_0(BG^1)\ar[d] \ar[r] & C_0(BG) \ar[d] \ar[r] &  0\ar[d] \ar[r] & 0 \\
&           0 & 0& 0
}
\end{equation}
where all rows are exact. Using the isomorphisms  $H_n(K)\cong H_n(C_\bullet (K))$, the long exact homology sequence 
of the pair $(BG,BG^1)$ and the long
exact homology sequence associated to the previous diagram, we obtain isomorphisms
$$
H_n(C_\bullet (BG,BG^1))\cong H_n(BG,BG^1)\, , \quad \forall n\, .
$$
Consider now the morphism $C_2(BG)\to \frac{R}{[F,R]}$ induced by sending any element of  $C_2(BG)$ to the homotopy class 
of its boundary. By \eqref{diag} we have isomorphisms   $C_n(BG,BG^1) \to C_n(BG)$, $n\geq 2$, and so we obtain a group homomorphism:
\begin{equation}\label{hopf+}
C_2(BG,BG^1)\to  \frac{R}{[F,R]}\, .
\end{equation}
By \cite[Satz I.]{Hopf} it follows that the kernel 
of \eqref{hopf+} is $\partial \left( C_3(BG,BG^1)\right)=\partial \left( C_3(BG)\right)$. So we get the homomorphism
\begin{equation}\label{hopf++}
H_2(BG,BG^1) \to \frac{R}{[F,R]}\, .
\end{equation}
By construction, the diagram in the statement commutes and \eqref{hopf++} is an isomorphism by the $5$-lemma.
\end{proof}

Let now $p\colon C \to C'$ be a $G$-covering branched at $y_1, \dots , y_d \in C'$. Fix once and for all a point $y_0 \in C' \setminus \{ y_1, \dots , y_d \}$
and a geometric basis $\ga_1, \ldots , \ga_d, \a_1, \b_1, \ldots , \a_{g'}, \b_{g'}$ of $\pi_1 (C'\setminus \{ y_1, \dots y_d \}, y_0)$. The monodromy of $p$ evaluated at the chosen geometric basis
gives the Hurwitz generating system $v \in HS(G;g',d)$,
well defined up to conjugation, which in turn determines $p$ (by Riemann's existence theorem). For $\Gamma =\Gamma_v$, the $\e$-invariant of $v$
is an element of 
$$
K_\Gamma:=\frac{R}{R_\Gamma}=\left( \frac{R}{[F,R]}\right)/_{\langle \langle \hat{a}\hat{b}\hat{c}^{-1}\hat{b}^{-1}\, | \, a\in \Gamma \, , \, ab=bc \rangle \rangle }\, .
$$
If $p$ is unramified, then, under the identification $ \frac{R\cap [F,F]}{[F,R]} \cong H_2(BG,\ZZ)$, 
$\e (v)\in \frac{R\cap [F,F]}{[F,R]}$ coincides with the image $Bp_*[C'] \in H_2(BG,\ZZ)$ of the fundamental class 
of $C'$ under the morphism induced in homology by a classifying map $BpÊ\colon C' \to BG$ of $p$.
We want to extend this topological interpretation of the $\e$-invariant to the ramified case.

By definition, any loop $\ga_i$ of the geometric basis consists of a path $\tilde{\ga}_i$ from $y_0$ to a point $z_i$ near $y_i$ and of a small loop around $y_i$.
Let $\Sigma$ be the Riemann surface (with boundary) obtained from $C'$ after removing the open discs surrounded by these small loops. 
Fix once and for all a CW-decomposition of $\Sigma$ as follows. The $0$-skeleton $\Sigma^0$ consists of  the  point $y_0$ and, for any 
$i=1, \dots , d$, the intersection $z_i$ between $\tilde{\gamma}_i$ and  the small circle of $\gamma_i$ around $y_i$.
The $1$-skeleton $\Sigma^1$ is given by  the geometric basis and the $2$-skeleton $\Sigma^2$ consists of one cell (see Figure 1).

\begin{figure}
	\centering
  \includegraphics[width=0.8\textwidth]{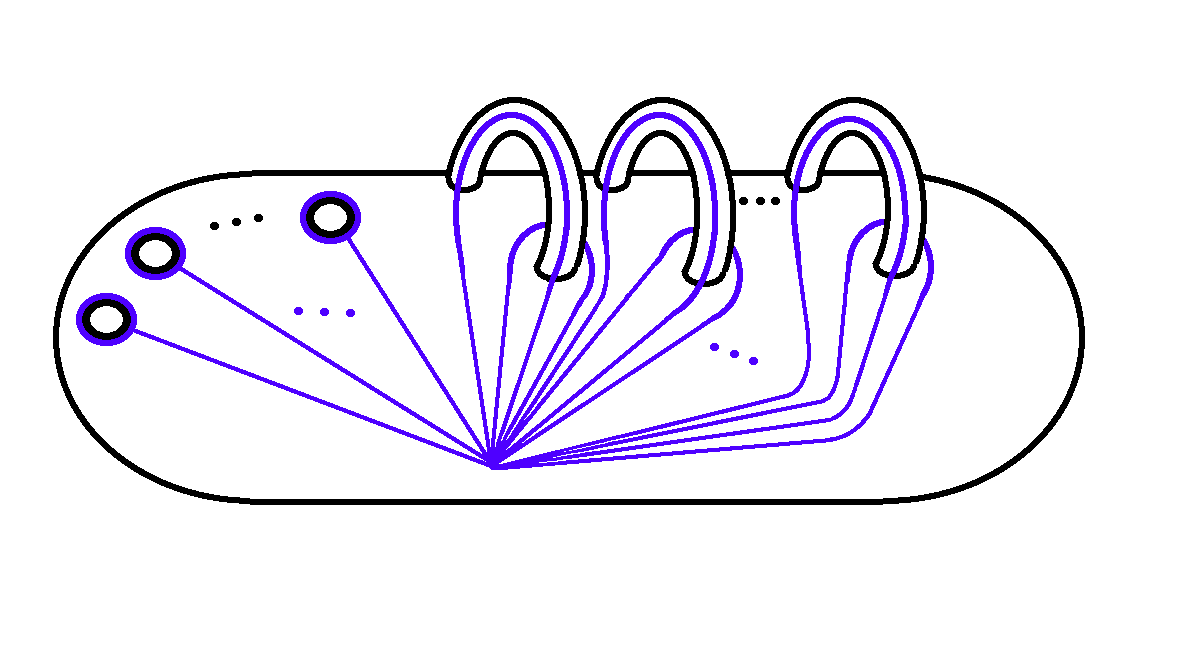}
  \caption{}
	\label{fig1}
\end{figure}

The restriction $p_\Sigma$ of $p\colon C \to C'$ to $p^{-1}(\Sigma)$ is an unramified $G$-covering of $\Sigma$ and hence corresponds to a continuous map
$Bp_\Sigma \colon \Sigma \to BG$, well defined up to homotopy. Let $Bp_1 \colon \Sigma \to BG$ be a cellular approximation of $Bp_\Sigma$.
Since $Bp_1$ can be regarded as a map of pairs $Bp_1 \colon (\Sigma, \partial \Sigma)\to (BG,BG^1)$, the push-forward of the fundamental 
(orientation) class $[\Sigma, \partial \Sigma]$ gives an element
$$
{Bp_1}_*  [\Sigma, \partial \Sigma] \in H_2(BG,BG^1)=\frac{R}{[F,R]} \, .
$$
This element depends on the chosen cellular approximation $Bp_1$ of $Bp_\Sigma$, but its image in $K_\Gamma$
does not (see Lemma \ref{bp1}). 


In order to compute ${Bp_1}_*  [\Sigma, \partial \Sigma]$ it is useful to recall the construction of  $Bp_1$.
As before, let 
$$
v= (c_1, \dots , c_d, a_1, b_1, \dots , a_{g'},b_{g'}) \in G^{d+2g'}\, 
$$
be a Hurwitz generating system of $p$ with respect to the given geometric basis.
Consider the relation 
$$
r:= \widehat{c_1} \cdot \ldots \cdot \widehat{c_d} \cdot \prod_{j=1}^{g'}[\widehat{a_j}, \widehat{b_j} ] \in \pi_1(BG^1) \, 
$$
and a continuous map $h_r\colon \Sigma^1 \to BG^1$ representing $r$. By the construction of $BG$ there is a $2$-cell $E_r$
attached along $h_r$. Choose a homeomorphism $E_r\cong \Sigma^2$.  So we get a continuous map $Bp_1 \colon \Sigma \to BG$.
To prove that $Bp_1$ is homotopic to $Bp$, we need to prove that the pull-back $Bp_1^*(EG \to BG)$ under $Bp_1$ of $EG \to BG$ is $p_\Sigma \colon p^{-1}(\Sigma) \to \Sigma$.
This follows from the fact that $EG \to BG$ is the universal cover and from the fact that the monodromy vector
of $Bp_1^*(EG \to BG)$ with respect to the given geometric basis  coincides with $v$ up to conjugation.
Moreover, by the proof of Lemma \ref{relhopf}, it follows that  under the isomorphism $H_2(BG,BG^1)\cong \frac{R}{[F,R]}$
\begin{equation}\label{bp*}
{Bp_1}_*  [\Sigma, \partial \Sigma]= \prod_1^d\widehat{c_i} \cdot \prod_1^{g'}[\widehat{a_j},\widehat{b_j}] \in \frac{R}{[F,R]}\, .
\end{equation}

\begin{lemma}\label{bp1}
Let $\Gamma =\Gamma_v \subset G$, where $v$ is the Hurwitz generating system of $p$
with respect to the given geometric basis.
Then the image of ${Bp_1}_* [\Sigma, \partial \Sigma]$ in $K_\Gamma$ does not depend 
on the cellular approximation $Bp_1$ of $Bp$. 
Denote this element $Bp_*[\Sigma, \partial \Sigma] \in K_\Gamma$. Then, by \eqref{bp*} we have:
$$
Bp_*[\Sigma, \partial \Sigma] =\e(v)\in K_\Gamma \, .
$$
\end{lemma}
\begin{proof}
Let $Bp_2 \colon (\Sigma,\partial \Sigma)\to (BG,BG^1)$ be another cellular approximation of $Bp_\Sigma$.
Then there exists an homotopy $F \colon I\times \Sigma \to BG$ such that $F(0,x)=Bp_1(x)$ and 
$F(1,x)=Bp_2(x)$, $\forall \, x\in \Sigma$, where $I=[0,1]$. Without loss of generality we can assume that $F$ is cellular 
with respect to the standard  CW-decomposition of $I\times\Sigma$  (\cite[ (4.7)]{Whitehead}):
$$
\left( I\times\Sigma \right)^n =\bigcup_{i=0}^n I^i \times \Sigma^{n-i} \, ,
$$
where $I$ is seen as a $1$-dimensional CW-complex with $I^0=\{ 0\} \cup \{ 1 \}$.

Consider now the chain homotopy 
$$
\varphi_n \colon C_n(BG) \to C_{n+1}(BG)
$$ 
between the chain maps ${Bp_1}_{\#}$ and ${Bp_2}_{\#}$ associated to $F$. Then we have (cf. \cite[(7.4.1)]{Massey}):
$$
{Bp_1}_{\#} - {Bp_2}_{\#} = \partial_{n+1}\circ \varphi_n + \varphi_{n-1}\circ \partial_n \, .
$$
>From this it follows that 
$$
{Bp_1}_* [\Sigma, \partial \Sigma] - {Bp_2}_* [\Sigma, \partial \Sigma] = \left[ \varphi_{1}(\partial \Sigma) \right] \in H_2(BG,BG^1) \, .
$$
Notice that 
$$
\left[ \varphi_{1}(\partial \Sigma) \right] = \left( F_{| I\times \partial \Sigma} \right)_*[I\times \partial \Sigma, \partial (I\times \partial \Sigma)] \in H_2(BG,BG^1) \, .
$$
The claim now follows from the fact that $I\times \partial \Sigma$ is the union of cylinders $I\times S^1$ and 
$$
\left( F_{| I\times \partial \Sigma} \right)_*[I\times S^1, \partial (I\times S^1)] = \hat{a}\hat{b}\hat{c}^{-1}\hat{b}^{-1} \, ,
$$
where $\hat{a}$ (resp. $\hat{c}$) is the image of the fundamental class of $\{ 0 \} \times S^1$ (resp. $\{ 1 \} \times S^1$) under $F$,
$\hat{b}$ is the image of $I\times \{ z_i \}$ under $F$, for some $z_i\in \Sigma^0$. 
Notice that, since $\left( F_{| I\times \partial \Sigma} \right)_*[I\times S^1, \partial (I\times S^1)] $ is the class of a $2$-cell,
$\hat{a}\hat{b}\hat{c}^{-1}\hat{b}^{-1}$ must be a relation for $G$.

\end{proof}

\section{The main Theorem}
In this Section we prove our main result. Roughly speaking, it says that for $g'$ sufficiently large 
the map given by the  $\hat{\e}$-invariant   \eqref{ehat} 
is injective, its image is independent of $g'$ and coincides with 
the classes of admissible $\nu$-types. We begin then  with the following
\begin{definition}\label{Nielsen}
Let $v\in HS(G;g',d)$ and let $\nu(v) \in \bigoplus_{\mathcal{C}}\ZZ\langle \mathcal{C}\rangle$ ($\mathcal{C}$ runs over the set 
$Conj (G)$ of conjugacy classes of $G$)
 be the vector whose $\mathcal{C}$-component
is the number of $v_j$, $j\leq d$, which belong to $\mathcal{C}$.\\
The map
\begin{equation}\label{nu}
\nu \colon HS(G;g',d) \to \bigoplus_{\mathcal{C}}\ZZ\langle \mathcal{C}\rangle 
\end{equation}
obtained in this way  induces a map 
$$
\tilde{\nu}\colon HS(G;g',d)/_{\Aut(G)} \to \left( \bigoplus_{\mathcal{C}} \ZZ \langle  \mathcal{C} \rangle \right)/_{\Aut(G)}
$$
which is $Map(g',d)$-invariant, therefore we get  a map 
$$
\hat{\nu}\colon \left( HS(G;g',d)/_{ \Aut(G)}\right)/_{Map(g',d)} \to \left( \bigoplus_{\mathcal{C}}\ZZ\langle \mathcal{C}\rangle \right)/_{\Aut(G)} \, .
$$
For any $v\in HS(G;g',d)$, we call $\hat{\nu}(v)$ the (unmarked) \textbf{$\nu$-type} of $v$ ( the marked version
$\nu(v)$ is called  shape in \cite{FV}).
\end{definition}

\begin{rem}\label{evsnu}
Let $v\in HS(G;g',d)$  and let $\Ga_v\subset G$ be the union of the conjugacy classes of the $v_j$, $j\leq d$.
The abelianization $G_{\Ga_v}^{ab}$ of $G_{\Ga_v}$ can be described as follows:
$$
G_{\Ga_v}^{ab}\cong \bigoplus_{\mathcal{C}\subset \Ga}\ZZ\langle \mathcal{C} \rangle \bigoplus_{g\in G\setminus \Ga_v}\ZZ \langle g \rangle \, ,
$$
where $\mathcal{C}$ denotes a conjugacy class of $G$.  Moreover $\nu(v)$ coincides with the vector whose $\mathcal{C}$-components
are the corresponding components of  the image in $G_{\Ga_v}^{ab}$
of $\e(v)\in G_{\Ga_v}$
under the natural homomorphism  $G_{\Ga_v} \to G_{\Ga_v}^{ab}$. It follows that $\nu$ in \eqref{nu}
factors as 
$$
\nu=A \circ \e 
$$
where 
$$
A \colon \coprod_{\Ga} K_\Ga \to \bigoplus_{\mathcal{C}}\ZZ\langle \mathcal{C} \rangle
$$
is induced from the abelianization $G_{\Ga} \to G_{\Ga}^{ab}$. To take into account the automorphisms of $G$
one defines similarly $\widehat{A}$ in such a way that $\widehat{\nu}=\widehat{A} \circ \widehat{\e}$.
 \end{rem}

\begin{definition}\label{admissible}
An element 
$$
\nu = (n_{\mathcal{C}})_\mathcal{C} \in \bigoplus_{\mathcal{C}\not= \{ 1 \}}\ZZ\langle \mathcal{C}\rangle 
$$
is {\bf admissible} if the following equality holds in $G^{ab}$ for its natural $\ZZ$-module structure:
$$
\sum_\mathcal{C} n_{\mathcal{C}} \cdot [\mathcal{C}] =0 \, 
$$
where $[\sC]$ denotes the class of any element of $\sC$ in the abelianization of $G$.

Accordingly we say that $\widehat{\nu} \in  \left( \bigoplus_{\mathcal{C}} \ZZ \langle  \mathcal{C} \rangle \right)/_{\Aut(G)}$
is admissible if it is the class of an admissible element.
\end{definition}

The main result of the paper is then the following
\begin{theo}\label{maintheo}
Let $G$ be a finite group. Then for any $d\in \NN$ there is an integer $s=s(d)$ such that 
$$
\hat{\e} \colon \left( HS(G; g', d)/_{\Aut(G)}\right)/_{Map(g',d)} \to \left( \coprod_\Ga K_\Ga \right)/_{\Aut(G)}\, 
$$
is injective for any $g'>s$. Moreover, for any $g'>s$,  the image of  $\hat{\e}$ is independent of $g'$
and it coincides with the pre-image under $\widehat{A}$ of the admissible $\widehat{\nu}$-types.
\end{theo}

\medskip

The main tool to understand the set $\left( HS(G; g', d)/_{\Aut(G)}\right)/_{Map(g',d)} $
is provided by the so-called  {\it stabilization}, which we are going
once more to review.
For any Hurwitz generating system
$$
v=(c_1, \ldots , c_d , a_1, b_1, \ldots , a_{g'}, b_{g'}) \in HS(G;g',d)
$$
define the ($h$-)stabilization $v^h$ of $v$ inductively by 
$$
v^0=v \, , \quad v^1=(c_1, \ldots , c_d , a_1, b_1, \ldots , a_{g'}, b_{g'}, 1,1)\, , \quad  v^h=(v^{h-1})^1 \, , \quad  \forall \, h \in \NN \, .
$$
Topologically, if $v$ corresponds to the monodromy
$\mu \colon \pi_1(C') \to G$, then $v^h$ corresponds to the monodromy
$\mu^h \colon \pi_1(C' \# C'') \to G$ obtained by extending $\mu$ by $1$ on the elements of $\pi_1(C'')$, where 
$C''$ is an oriented surface of genus $h$ and $C' \# C''$ is the connected sum of $C'$ with $C''$.  \\

It is easy to see that stabilization satisfies the following properties:
the $\e$-invariant does not change under stabilization,
\begin{equation}\label{einvstab}
\e (v) = \e (v^h)\, , \quad \forall \, v \, , h \, ;
\end{equation}
it respects the equivalence relation given by the actions of  $\Aut (G)$ and ${Map(g',d)}$,  therefore we get maps
\begin{align}\label{stabmod}
\left( HS(G;g',d)/_{\Aut (G)} \right) /_{Map(g',d)} & \to \left( HS(G;g'+h,d)/_{\Aut (G)} \right) /_{Map(g'+h,d)} \, , \quad \forall \, g', h, d \\
[v] & \mapsto [v^h] \, . \nonumber
\end{align}

Moreover, the $\e$-invariant is stably a fine invariant. This is the content of the following theorem that 
extends to the ramified case the analogous result of Livingston \cite{Livingston}
for non-ramified group actions (see also \cite{DT, Zimmer}).
\begin{theo}\label{main}
Let $G$ be a finite group and let $v , w \in HS(G;g',d)$ such that $\nu(v)=\nu(w)$ (in particular 
$ \Gamma_{v}=\Gamma_{w}=\Gamma$). If
$$
 \quad \e(v)=\e(w)\in K_\Gamma \, ,
$$
then $\exists h \in \NN$ such that the classes of  $v^h$ and $w^h$ in $\left( HS(G;g'+h,d)/_{\Aut(G)} \right)/_{Map(g'+h,d)}$
coincide.
\end{theo}
\noindent We postpone the proof of this theorem to the next section. Here we use it to give a proof of  Thm. \ref{maintheo}.

\medskip

\noindent {\it Proof of Thm. \ref{maintheo}.}
First of all we have that, for $g' \geq |G|$, any $v\in HS(G;g'+1,d)$ is a stabilization, that is:
the map
\begin{align*}
\varsigma_{g'} \colon HS(G;g',d))/_{Map(g',d)}  & \to HS(G;g'+1,d))/_{Map(g' + 1, d)}  \\
{\rm induced \  by } \ 
v & \mapsto v^1 
\end{align*}
is surjective.
This follows from the proof of  \cite[Prop. 6.16]{DT}, where the result is stated 
for free group actions, but the same proof works also for non-free actions.  \\
As a consequence we have that  \eqref{stabmod} is surjective for $g' \geq |G|$.

Since the sets $\left( HS(G;g',d)/_{\Aut (G)} \right) /_{Map(g',d)} $
are finite sets, there  exists an integer 
\footnote{  to prove that $s$ can be explicitly given it would suffice to show that  once $\varsigma_{g'} $ is bijective, then also $\varsigma_{g''} $ is
bijective for $ g '' > g'$. }
$s=s(d)$ such that 
\eqref{stabmod} is bijective for any $g'>s$, and $h\geq 1$.

Let now $g' >s$ and let  $v,w\in HS(G;g',d)$ with $\e(v)=\e(w)$. From Thm. \ref{main} there exists  $h$ with $[v^h]=[w^h] \in 
\left( HS(G;g'+h,d)/_{\Aut (G)} \right) /_{Map(g'+h,d)}$. Since stabilization is bijective in this range we get
$[v]=[w] \in \left( HS(G;g',d)/_{\Aut (G)} \right) /_{Map(g',d)}$. Hence $\widehat{\e}$ is injective for $g'>s$.\\
The fact that  ${\rm Im} (\widehat{\e})$ does not depend on $g'$ now follows  from \eqref{einvstab}.

We now prove that ${\rm Im}(\widehat{\e})$ is the preimage under $\widehat{A}$ of the admissible $\nu$-types,
for any $g'>s$. First notice that for any admissible $\nu =(n_{\mathcal{C}})_{\mathcal{C}} \in \bigoplus_{\mathcal{C}}\ZZ \langle \mathcal{C}\rangle$
there exists
$v\in HS(G;s+1,d)$ with $\nu(v)=\nu$. Indeed for any $(c_1, \ldots , c_d)$ with $\nu(c_1, \ldots , c_d)=\nu$
the condition that $\nu$ is admissible implies that $c_1 \cdot \ldots \cdot c_d \in [G,G]$, hence it is a product of commutators,
$$
c_1 \cdot \ldots \cdot c_d=\left( \prod_{j=1}^r [a_j, b_j] \right)^{-1} \, ,
$$
for some $r\in \NN$. If $c_1, \ldots , c_d, a_1, b_1, \ldots , a_r, b_r$ do not generate $G$,
we add pairs of the form $(g, 1)$, so we obtain a Hurwitz generating
system $w\in HS(G;r,d)$ with $\nu(w)=\nu$.
Moreover we can assume that $r>s$. Since $\varsigma_{g'}$ is surjective for  $g'>s$, we get a Hurwitz system $v\in HS(G;s+1,d)$
whose stabilization is $Map(g',d)$-equivalent to $w$ and the claim follows.  \\
Finally we prove that for any $\xi \in K_{\Ga}$, $\Ga= \Ga_v $, with $A(\xi)=\nu$ there exists $w\in HS(G;s+1,d)$ 
with $\e(w)=\xi$.
Since $A(\xi)=A(\e(v))=\nu$, $\e(v)^{-1} \cdot \xi \in H_{2,\Ga}(G)$. By \eqref{h2toh2ga} there exists $\eta \in H_2(G,\ZZ)$ which maps
to $\e(v)^{-1} \cdot \xi$. Since bordism is the same as homology in dimension $2$ (cf. also \cite[Thm. 6.20]{DT}),
there is a Hurwitz system $v' \in HS(G;h,0)$ such that $\e (v')=\eta \in H_2(G,\ZZ)$. Let $v'' \in HS(G;s+1+h,d)$ be the system
whose first $d+2(s+1)$ components coincide with those of $v$ and the last $2h$ components are those of $v'$, 
then we have: $\e(v'')=\xi$. By \eqref{einvstab} the system $w\in HS(G;s+1,d)$ that maps to $v''$  under $\varsigma_{s+h} \circ \ldots \circ \varsigma_{s+1}$ 
satisfies $\e(w)=\xi$.
This concludes the proof of the theorem. 
 
\qed

\section{Proof of  Theorem \ref{main}}

Let $v,w \in HS(G,g',d)$ be two Hurwitz generating systems with $\nu(v)=\nu(w)$ and $\e(v)=\e(w)$.
If
\begin{align*}
 v =& (v_1,\dots , v_d \, ; \, v_{d+1},v_{d+2},\dots , v_{d+2g'-1}, v_{d+2g'}) \quad \mbox{and} \\
w =& (w_1,\dots , w_d \, ; \, w_{d+1},w_{d+2},\dots , w_{d+2g'-1}, w_{d+2g'}) \, ,
\end{align*}
then, without loss of generality (using braid group moves on the first $d$ components)  we may assume that   $v_i$ is conjugate to $w_i$, $i=1, \dots , d$, and that the following 
equality holds:
\begin{equation}\label{modrg}
 \prod_{i=1}^d \widehat{w_i}\prod_{j=1}^{g'}[\widehat{w_{d+2j-1}},\widehat{w_{d+2j}}] \quad \equiv \quad  
 \prod_{i=1}^d \widehat{v_i}\prod_{j=1}^{g'}[\widehat{v_{d+2j-1}},\widehat{v_{d+2j}}] \quad    (\!\!\!\!\! \mod R_\Gamma) \, . 
\end{equation}
Let us rewrite equation \eqref{modrg} modulo $[F,R]$. This means that there are relations $\widehat{x_\ell}\widehat{y_\ell}\widehat{z_\ell}^{-1}\widehat{y_\ell}^{-1}\in R$, $\ell =1, \dots , N$, and with $x_{\ell} \in \Ga$, such that 

\begin{equation}\label{modfr}
 \prod_{i=1}^d \widehat{w_i}\prod_{j=1}^{g'}[\widehat{w_{d+2j-1}},\widehat{w_{d+2j}}] \equiv
 \prod_{i=1}^d \widehat{v_i}\prod_{j=1}^{g'}[\widehat{v_{d+2j-1}},\widehat{v_{d+2j}}]
 \prod_{\ell=1}^N \left( \widehat{x_\ell}\widehat{y_\ell}\widehat{z_\ell}^{-1}\widehat{y_\ell}^{-1} \right)^{\pm 1}    (\!\!\!\!\! \mod [F,R]) \, . 
\end{equation}

\begin{prop}\label{sc}
Let $v, w\in HS(G,g',d)$ be two Hurwitz generating systems. Assume  $v_i =w_i$, $\forall i=1, \dots , d$,
$ev(\hat{v})\equiv ev(\hat{w})$ $(\!\!\!\! \mod [F,R])$ (i.e., $N=0$ in \eqref{modfr}), and $G=\langle v_{d+1}, \dots , v_{d+2g'}\rangle = 
\langle w_{d+1}, \dots , w_{d+2g'}\rangle$. Then $v$ and $w$ are  stably equivalent.
\end{prop}
\begin{proof}

We have
$$
 ev(\hat{w})=\prod_{i=1}^d \widehat{v_i}\prod_{j=1}^{g'}[\widehat{w_{d+2j-1}},\widehat{w_{d+2j}}] \equiv ev(\hat{v})=
 \prod_{i=1}^d \widehat{v_i}\prod_{j=1}^{g'}[\widehat{v_{d+2j-1}},\widehat{v_{d+2j}}]
 \quad  (\!\!\!\!\! \mod [F,R])
$$
 hence 
\begin{equation}\label{d0sc}
\prod_{j=1}^{g'}[\widehat{w_{d+2j-1}},\widehat{w_{d+2j}}] \equiv \prod_{j=1}^{g'}[\widehat{v_{d+2j-1}},\widehat{v_{d+2j}}] \quad (\!\!\!\!\! \mod [F,R]) \, .
\end{equation}
This means that the right hand side of \eqref{d0sc} differs from the left hand side by a product of commutators in $[F,R]$.
Now, by Livingston's theorem \cite{Livingston} (cf. the algebraic proof in \cite{Zimmer}), it follows that we can realize these commutators
 by adding handles with trivial monodromies and acting with the mapping class group $Map_{g'}$.  
The mapping classes that are used in this procedure can be seen to be the restriction of mapping classes in $Map(g',d)$ 
that act as the identity on the first $d$ components  of the Hurwitz systems (see \cite{Zimmer}, 2.1--2.5). 
The proposition now follows from the similar result in the 
unramified case.
\end{proof}

We prove Theorem \ref{main} by showing that $v$ and $w$ are stably equivalent to Hurwitz generating systems
 for which the hypotheses of Proposition \ref{sc}
are satisfied. This is achieved through two reduction steps: first we prove that, after stabilization, we can  assume that $v_i=w_i$,
$i=1, \dots , d$, 
and $G=\langle v_{d+1}, \dots , v_{d+2g'}\rangle = 
\langle w_{d+1}, \dots , w_{d+2g'}\rangle$; then we stabilize further to obtain $N=0$ in \eqref{modfr}. 

\subsection*{1st step: reduction to the case $v_i=w_i$, $i=1, \dots , d$, $G=\langle v_{d+1}, \dots , v_{d+2g'}\rangle = 
\langle w_{d+1}, \dots , w_{d+2g'}\rangle$}

\begin{prop}\label{1st step}
Let 
$$
v=(c_1, \dots , c_d; a_1, b_1, \dots , a_{g'}, b_{g'})\in HS(G;g',d)
$$
be a Hurwitz generating system and let $g_1, \dots , g_d \in G$. Set $c_i'=g_ic_ig_i^{-1}$.  Then,  there exists $\varphi \in Map(g'+d,d)$ such that 
\begin{equation}\label{v2tilde}
\varphi \cdot v^d \quad = \quad (c_1', \dots , c_d'\, ; \, \lambda_1, \mu_1, \dots , \lambda_d, \mu_d,  a_{1},, b_1, \dots , a_{g'},  {b}_{g'})\, .
\end{equation}
Here $v^d$ is obtained from $v$ by adding $d$ handles with trivial monodromies.
 Precise formulas for the $\lambda$'s and $\mu$'s are given below.
\end{prop}
 We use the following 
\begin{lemma}\label{zimmer+}
For any
$$
(c_1, \dots , c_d \, ; \, a_1, b_1, \dots , a_{g'}, b_{g'}) \in HS(G;g',d)
$$
and for any $x \in G$, we have:
$$
(c_1, \dots , c_d \, ; \, 1,1, a_1, b_1, \dots , a_{g'}, b_{g'}) \approx  (c_1, \dots , c_d \, ; \, x,1, a_1, b_1, \dots , a_{g'}, b_{g'}) \, .
$$
\end{lemma}
\begin{proof}

A direct computation shows that, for any  automorphism $\varphi$ of the form 2.1--2.5 in \cite{Zimmer}, the map
\begin{align*}
\overline{\varphi} & \colon \langle c_1,\dots , c_d;a_1,\dots , b_{g'}| \prod_1^dc_i\prod_1^{g'}[a_j,b_j]=1\rangle \to  
\langle c_1,\dots , c_d;a_1,\dots , b_{g'}| \prod_1^dc_i\prod_1^{g'}[a_j,b_j]=1\rangle  \\
& c_i \mapsto c_i ,  \quad  a_j \mapsto \varphi (a_j), \quad  b_j \mapsto \varphi (b_j)
\end{align*}
defines an automorphism, so $\overline{\varphi}\in Map\left( C',\mathcal{B};\{ y_0\} \right)$. Here  $\sB$ is  the branch locus of $ C \ra C'$ and 
$Map\left( C',\mathcal{B};\{ y_0\} \right)$ is the group of isotopy classes of diffeomorphisms
$f\colon C' \to C'$ such that $f$ preserves the orientation, $f(\mathcal{B})=\mathcal{B}$, and $f(y_0)=y_0$. 

By \cite{Zimmer}  Lemma 2.6 the claim is true for any $x \in \langle a_1, \dots , b_{g'}\rangle$, therefore it remains to prove that
$$
(c_1, \dots , c_d \, ; \, x,1, a_1, b_1, \dots , a_{g'}, b_{g'}) \approx (c_1, \dots , c_d \, ; \, c_i x,1, a_1, b_1, \dots , a_{g'}, b_{g'}) ,
$$
for any $i=1,\dots , d$. Using the braid group, it is enough to prove the above equivalence when $i=d$:
the result follows 
as a direct consequence of   Proposition 6.2 of \cite{CLP12}, (i),  with $\ell =1$ (see figure 2),  
yielding the following useful transformation which leaves all the components of the Hurwitz vector unchanged except for

\begin{multline}\label{CLP}
a_1 \mapsto  c_d a_1 ,  \  \  c_d  \mapsto  (c_d a_1 b_1 a_1^{-1}) c_d   (c_d a_1 b_1 a_1^{-1})^{-1} \\
\Leftrightarrow
v_{d+1}  \mapsto  v_d v_{d+1}  ,  \  \  v_d  \mapsto  
g v_d g^{-1} \\
g : = (v_d v_{d+1} v_{d+2} v_{d+1}^{-1})  .
\end{multline}
\end{proof}

\begin{figure}
	\centering
  \includegraphics[width=0.8\textwidth]{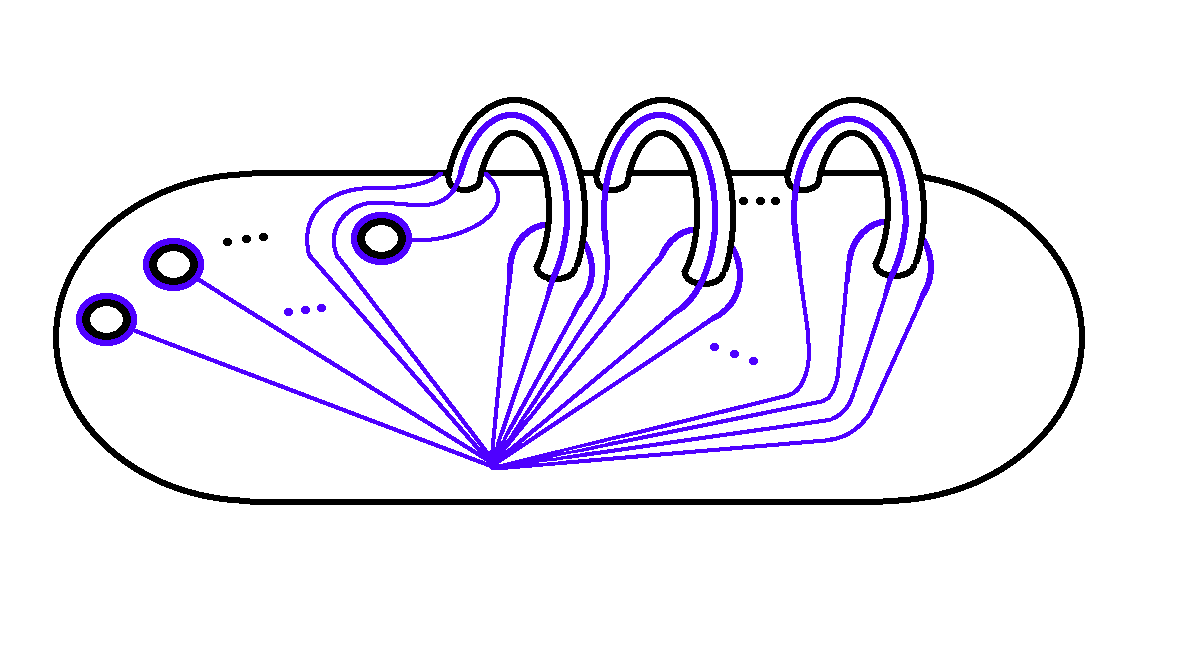}
  \caption{}
	\label{fig2}
\end{figure}

\begin{proof} (of Proposition \ref{1st step})
$\varphi$ is the composition of the mapping classes corresponding to the following steps 2 -- 7.
Set $h_i=c_i^{-1}g_i$ and perform the following operations. 
\begin{itemize}
\item[1.]
Add a trivial handle to $v$ obtaining: $ (c_1,\dots , c_d \, ; \,  1,1, a_1,b_1,\dots , a_{g'}, b_{g'})$.
\item[2.]
Bring $c_1$ to the $d$-th position using the braid group: 
$$
(c_1,\dots , c_d \, ; \,  1,1, a_1,b_1,\dots , a_{g'}, b_{g'}) \approx (c_1c_2c_1^{-1},\dots , c_1 \, ; \,  1,1, a_1,b_1,\dots , a_{g'}, b_{g'})\, .
$$
\item[3.]
Apply Lemma \ref{zimmer+} with $x=h_1$:
$$
(c_1c_2c_1^{-1},\dots , c_1 \, ; \,  1,1, a_1,b_1,\dots , a_{g'}, b_{g'}) \approx (c_1c_2c_1^{-1},\dots , c_1 \, ; \,  h_1,1, a_1,b_1,\dots , a_{g'}, b_{g'})\, .
$$
\item[4.]
Change $h_1,1$ to $h_1,h_1$ according to the automorphism
of \cite[2.1.b)]{Zimmer}.
\item[5.]
Apply again Proposition 6.2 of \cite{CLP12} (i) with $\ell =1$ 
(see \eqref{CLP} and  figure 2 again):
\begin{align*}
& (c_1c_2c_1^{-1},\dots , c_1 \, ; \,  h_1,h_1, a_1,b_1,
\dots , a_{g'}, b_{g'}) \\
\approx &
(c_1c_2c_1^{-1},\dots , 
(c_1h_1h_1h_1^{-1})c_1 (c_1h_1h_1h_1^{-1})^{-1} \, ; \,  
c_1h_1, h_1 , a_1,b_1,\dots , a_{g'}, b_{g'}) \\
=& (c_1c_2c_1^{-1},\dots , g_ic_1g_i\inv \, ; \,  c_1h_1, h_1 , a_1,b_1,\dots , a_{g'}, b_{g'}) \, \\
=& (c_1c_2c_1^{-1},\dots , c_1' \, ; \,  c_1h_1, h_1 , a_1,b_1,\dots , a_{g'}, b_{g'}) \, .
\end{align*}
\item[6.]
Use the braid group to move the last monodromy to the first position:
\begin{eqnarray*}
&& 
(c_1c_2c_1^{-1},\dots , c_1' \, ; \,  c_1h_1, h_1 , a_1,b_1,\dots , a_{g'}, b_{g'})\\
& \approx & 
(c_1', \overline{c_2}, \dots , \overline{c_d} \, ; \,  c_1h_1, h_1 ,a_1,b_1,\dots , a_{g'}, b_{g'})
\end{eqnarray*}
where  $\overline{c_i}$ is a conjugate of $c_i$, $\forall i$.
\item[7.]
Repeat the steps above for $\overline{c_2}$ and so on.
\end{itemize}
\end{proof}

\begin{rem}
The condition $G=\langle v_{d+1}, \dots , v_{d+2g'}\rangle = 
\langle w_{d+1}, \dots , w_{d+2g'}\rangle$ in Prop. \ref{sc} can be achieved by using Lemma \ref{zimmer+}.
\end{rem}
\subsection*{2nd step: reduction to the case $N=0$}
Let $v$ and $w$ be Hurwitz generating systems as in the beginning of the section.
By the 1st step we  assume that $v_i=w_i$, $i=1, \dots , d$,
and that $G=\langle v_{d+1}, \dots , v_{d+2g'}\rangle = \langle w_{d+1}, \dots , w_{d+2g'}\rangle$.
By hypothesis we have:
\begin{equation}\label{eq2nd1}
ev(\hat{w}) \equiv ev(\hat{v})\prod_{\ell=1}^N \left( \widehat{x_\ell}\widehat{y_\ell}\widehat{z_\ell}^{-1}\widehat{y_\ell}^{-1} \right)^{\s_\ell}    (\!\!\!\!\! \mod [F,R]) \, ,
\end{equation}
where $x_\ell y_\ell z_\ell^{-1}y_\ell^{-1}=1$, $\s_\ell =\pm 1$, $x_\ell , z_\ell \in \Gamma$.

The main result of this subsection is the following
\begin{prop}\label{2nd step}
Let $v,w \in HS(G;g',d)$ be Hurwitz generating systems with $\nu(v)=\nu(w)$ and $\e(v)=\e(w)$.
Assume further  $v_i=w_i$, $i=1,\dots , d$, and  
$G=\langle v_{d+1}, \dots , v_{d+2g'}\rangle = \langle w_{d+1}, \dots , w_{d+2g'}\rangle$. 
Then there exist $h\in \NN$, $\varphi, \psi \in Map(g'+h,d)$ such that 
\begin{eqnarray*}
ev(\widehat{\psi\cdot w^h}) &\equiv & ev(\widehat{\varphi \cdot v^h}) \quad (\!\!\!\! \mod [F,R])\, ,   \mbox{and}\\
(\psi\cdot w^h)_i&=&(\varphi \cdot v^h)_i \, , \quad \forall i=1, \dots , d \, .
\end{eqnarray*}
\end{prop}

To prove Proposition \ref{2nd step}, we first rewrite $\prod_{\ell=1}^N \left( \widehat{x_\ell}\widehat{y_\ell}\widehat{z_\ell}^{-1}\widehat{y_\ell}^{-1} \right)^{\s_\ell} $
in \eqref{eq2nd1} as a product of commutators of the form $[\hat{\eta}, \hat{\x}]$ where $\x \in \Gamma$ and $[\eta,\xi]=1$.
To achieve this, we use the following identities.

\begin{lemma}\label{=modfr}
Let $x,y,z, y_1, z_1 \in G$ be such that 
$$
xyz^{-1}y^{-1}={z}{y_1}{z_1}^{-1}{y_1}^{-1} =1 \, .
$$
Then the  following congruences hold, where as usual $\hat{x}, \hat{y}, \hat{z},\widehat{y_1}, \widehat{z_1} \in F$
are the tautological lifts of $x,y,z, y_1, z_1$.
\medskip
\begin{itemize}
\item[(i)]
$ \widehat{x}\widehat{y}\widehat{z}^{-1}\widehat{y}^{-1}\equiv \widehat{y}^{-1}\widehat{x}\widehat{y}\widehat{z}^{-1} \quad (\!\!\!\!\mod [F,R])$.
\medskip
\item[(ii)]
$\left( \widehat{x}\widehat{y}\widehat{z}^{-1}\widehat{y}^{-1} \right)^{-1}=
\widehat{y}\widehat{z}\widehat{y}^{-1}\widehat{x}^{-1}\equiv
\widehat{z}\widehat{y}^{-1}\widehat{x}^{-1}\widehat{y}\quad (\!\!\!\!\mod [F,R])$.
\medskip
\item[(iii)]
$(\hat{x}\hat{y}\hat{z}^{-1}\hat{y}^{-1})(\hat{z}\widehat{y_1}\widehat{z_1}^{-1}\widehat{y_1}^{-1})\equiv \hat{x}\widehat{yy_1}\widehat{z_1}^{-1}\widehat{yy_1}^{-1}
\quad (\!\!\!\!\mod [F,R])$.
\medskip
\item[(iv)]
$\hat{x}\hat{y}^{\s}\hat{z}^{-1}\hat{y}^{-\s} \equiv \hat{x}\widehat{y^{\s}}\hat{z}^{-1}\widehat{y^{\s}}^{-1}
\quad (\!\!\!\!\mod [F,R])$, where $\s =\pm 1$.
\end{itemize}
\end{lemma}

\begin{proof}
(i) and (ii) follows from the fact that, if $ r \in R$, then $$r \equiv \widehat{y}^{-1}  r \widehat{y}  \ (\!\!\!\mod [F,R]).$$

For (iv) we will use that $[F,R]$ is a normal subgroup of $F$.

\medskip

(iii) Using (i) we have:
$$
(\hat{x}\hat{y}\hat{z}^{-1}\hat{y}^{-1})(\hat{z}\widehat{y_1}\widehat{z_1}^{-1}\widehat{y_1}^{-1})\equiv 
(\hat{y}^{-1}\hat{x}\hat{y}\hat{z}^{-1})(\hat{z}\widehat{y_1}\widehat{z_1}^{-1}\widehat{y_1}^{-1}) = 
\hat{y}^{-1}\hat{x}\hat{y}\widehat{y_1}\widehat{z_1}^{-1}\widehat{y_1}^{-1} \quad (\!\!\!\! \mod [F,R]) \, .
$$
Moreover the above element is in $R$, since it is congruent to a product of two elements in $R$ modulo $[F,R]$. Hence we have by the usual token:
$$
\hat{y}^{-1}\hat{x}\hat{y}\widehat{y_1}\widehat{z_1}^{-1}\widehat{y_1}^{-1} \equiv  \hat{x}\hat{y}\widehat{y_1}\widehat{z_1}^{-1}\widehat{y_1}^{-1} \hat{y}^{-1} 
 \quad (\!\!\!\! \mod [F,R]) \, ,
$$
in fact the right hand side is just the conjugate of the left hand side by $\widehat{y}$.
Finally
$$
(\hat{x}\hat{y}\widehat{y_1}\widehat{z_1}^{-1}\widehat{y_1}^{-1} \hat{y}^{-1})(\hat{x}\widehat{yy_1}\widehat{z_1}^{-1}\widehat{yy_1}^{-1})^{-1}=
(\hat{x}\hat{y}\widehat{y_1})[\widehat{z_1}^{-1}, \widehat{y_1}^{-1} \hat{y}^{-1}\widehat{yy_1}](\hat{x}\hat{y}\widehat{y_1})^{-1} \in [F,R] \, .
$$

\medskip

(iv) We have: $(\hat{x}\hat{y}^{\s}\hat{z}^{-1}\hat{y}^{-\s})( \hat{x}\widehat{y^{\s}}\hat{z}^{-1}\widehat{y^{\s}}^{-1})^{-1}= (\hat{x}\hat{y}^\s)[\hat{z}^{-1}, \hat{y}^{-\s}\widehat{y^\s}]
(\hat{x}\hat{y}^\s)^{-1}\in [F,R]$.

\end{proof}

\begin{lemma}\label{2ndstepl1}
Let $v, w$ be as in Proposition \ref{2nd step}. 
Then there exist  $M\in \NN$ and, for $m=1, \dots , M$, elements $\x_m \in \Gamma$ and $\eta_m \in G$ with $[\x_m, \eta_m]=1$,  such that 
$$
ev(\hat{w}) \equiv ev(\hat{v})\prod_{m=1}^M [ \widehat{\x_m}, \widehat{\eta_m}] \quad    (\!\!\!\!\!\! \mod [F,R]) \, .
$$
\end{lemma}

\begin{proof}
Using Lemma \ref{=modfr} (ii) and (iv), rewrite  \eqref{eq2nd1} as
\begin{equation}\label{eq2nd2}
ev(\hat{w}) \equiv ev(\hat{v})\prod_{\ell=1}^N \widehat{a_\ell}\widehat{b_\ell}\widehat{c_\ell}^{-1}\widehat{b_\ell}^{-1}    (\!\!\!\!\! \mod [F,R]) \, ,
\end{equation}
where $a_\ell =x_\ell$, $b_\ell=y_\ell$, $c_{\ell}=z_\ell$, if $\s_\ell =1$, and 
$a_\ell =z_\ell$, $b_\ell=y_\ell^{-1}$, $c_{\ell}=x_\ell$, if $\s_\ell =-1$.

Consider the image of  \eqref{eq2nd2} in the abelianized group $F^{ab}$. Since $v_i=w_i$, $i=1, \dots , d$,  we get:
$$
\prod_{\ell =1}^N \widehat{a_\ell} \widehat{c_\ell}^{-1}=1   (\!\!\!\!\! \mod [F,F]) \, .
$$
Hence there exists a permutation $\tau \in \frak{S}_N$, such that $c_\ell =a_{\tau (\ell)}$,  for any $\ell=1, \dots , N$. 

Let us treat first the case where $\tau$ is a cycle of length $N$: then the set $\{ a_{\tau^k(1)}\, | \, k\in \NN \} =\{ a_1, \dots , a_N\}$ and,
since the product of the factors  $ (\!\!\!\!\! \mod [F,R])$ is independent of the order,
\begin{eqnarray*}
\prod_{\ell=1}^N \widehat{a_\ell}\widehat{b_\ell}\widehat{c_\ell}^{-1}\widehat{b_\ell}^{-1} \equiv 
(\widehat{a_1}\widehat{b_1}\widehat{c_1}^{-1}\widehat{b_1}^{-1})
\prod_{k=1}^{N-1} \left( \widehat{c_{\tau^{k-1}(1)}} \widehat{b_{\tau^k(1)}}\widehat{c_{\tau^k(1)}}^{-1} \widehat{b_{\tau^k(1)}}^{-1} \right)   (\!\!\!\!\! \mod [F,R])\, .
\end{eqnarray*}
Setting  $\x_1:=a_1$ and $\eta_1:= \prod_{k=0}^{N-1}b_{\tau^k(1)}$,
we obtain:
$$
\prod_{\ell=1}^N \widehat{a_\ell}\widehat{b_\ell}\widehat{c_\ell}^{-1}\widehat{b_\ell}^{-1} \equiv \,
[\widehat{\x_1}, \widehat{\eta_1}]  \quad (\!\!\!\! \mod [F,R]) \, ,
$$
where the equivalence follows from Lemma \ref{=modfr}, (iii).
Since the left hand side of the above equivalence 
is in $R$, it follows that $[\x_1, \eta_1]=1$.

The general case, where $\tau$ is a product of cycles of length $<N$, follows by induction.

\end{proof}

To complete the proof of Proposition \ref{2nd step}, we need to know how $ev(\hat{v})$ changes
under the action of the mapping class group, modulo $[F,R]$. Notice in fact  that $ev(\hat{v})$ is $Map(g',d)$-invariant only modulo $R_\Gamma$. 

\begin{lemma}\label{2nd step l3}
Let $v\in HS(G;g',d)$ and let $\varphi \in Map(g',d)$. Then we have:
\begin{itemize}
\item[(i)]
$ev(\widehat{\varphi \cdot v})=ev(\hat{v})$, if $(\varphi \cdot v)_i=v_i$, $\forall i=1, \dots , d$;
\item[(ii)]
$ev(\widehat{\varphi \cdot v})\equiv ev(\hat{v})\left( \widehat{v_{i+1}}^{-1}  \widehat{v_{i}}^{-1}  \widehat{v_{i+1}} \widehat{v_{i+1}^{-1}v_iv_{i+1}}\right)
\quad (\!\!\!\! \mod [F,R])$, if $\varphi$ is the half-twist $\s_i(v_i,v_{i+1})=(v_{i+1}, v_{i+1}^{-1}v_iv_{i+1})$;
\item[(iii)]
$ev(\widehat{\varphi \cdot v})\equiv ev(\hat{v}) \left(  \widehat{g}\widehat{v_d}^{-1}\widehat{g}^{-1}\widehat{gv_dg^{-1}}  \right) \:
(\!\!\!\! \mod [F,R])$,
if $\varphi$ is as in  \eqref{CLP} (Proposition 6.2 i) of \cite{CLP12} with $\ell=1$)
and then $g=v_dv_{d+1}v_{d+2}v_{d+1}\inv$.
\end{itemize}
\end{lemma}
\begin{proof}
The proof is similar to that of Proposition 3.6  of \cite{CLP12} (invariance of the $\e$-invariant under $Map(g',d)$).

Let $\alpha \colon \frac{F}{[F,R]}\to G$ be the morphism  $\hat{g}\mapsto g$. Then $\ker (\alpha)=\frac{R}{[F,R]}$
is central in $\frac{F}{[F,R]}$.
We have that 
$$
\a ((\widehat{\varphi \cdot v})_i)=({\varphi \cdot v})_i = \a ((\varphi \cdot \hat{v})_i) \, , \quad \forall i=1, \dots , d+2g' \, .
$$
Hence, there are $\z_i \in \frac{R}{[F,R]}$ such that 
$$
(\widehat{\varphi \cdot v})_i = (\varphi \cdot \hat{v})_i\cdot \z_i \, , \quad \forall i=1, \dots , d+2g' \, .
$$
Since $\frac{R}{[F,R]} \leq \frac{F}{[F,R]}$ is central, we can replace $(\widehat{\varphi \cdot v})_i$ with 
$(\varphi \cdot \hat{v})_i$, $\forall i=d+1, \dots , d+2g'$, in the expression for $ev(\widehat{\varphi \cdot v})$.
This is enough to prove the claim if  $\varphi$ acts as the identity on $v_i$, $i=1, \dots , d$. Indeed, in this case, we have: 
$$
ev(\widehat{\varphi \cdot v}) = ev ( {\varphi \cdot \hat{v}} ) = ev(\hat{v}) \, ,
$$
by the fact that evaluation is invariant under mapping classes.  This proves (i).

(ii) follows by a direct computation.

Finally, consider the case where $\varphi$ is as defined in 
\eqref{CLP} (Proposition 6.2 i) of \cite{CLP12} with $\ell=1$).
We have that $(\varphi \cdot v)_d=gv_dg^{-1}$ and 
$(\varphi \cdot v)_i=v_i$, $i=1, \dots , d-1$, where $g\in G$
is given above.
Let $x \in \frac{F}{[F,R]}$ such that 
$(\varphi \cdot \hat{v})_d = x\widehat{v_d}x^{-1}$. Since $\a(x)=g$, there is 
$\eta \in \frac{R}{[F,R]}$ -- hence central --
such that $x=\widehat{g}\eta$. It follows that
\begin{eqnarray*}
(\varphi \cdot \hat{v})_d = (\widehat{g}\eta)\widehat{v_d}(\widehat{g}\eta)^{-1}=\widehat{g}\widehat{v_d}\widehat{g}^{-1} \,  .
\end{eqnarray*}
Now notice that 
\begin{eqnarray*}
(\widehat{\varphi \cdot v})_d &=& \widehat{gv_dg^{-1}} \\
& = & \widehat{g}\widehat{v_d}\widehat{g}^{-1} \left( \widehat{g}\widehat{v_d}^{-1}\widehat{g}^{-1}\widehat{gv_dg^{-1}} \right) \\
&=&  (\varphi \cdot \hat{v})_d \left( \widehat{g}\widehat{v_d}^{-1}\widehat{g}^{-1}\widehat{gv_dg^{-1}} \right) \,  ,
\end{eqnarray*}
from which the claim follows.
\end{proof}

\bigskip

\noindent \textit{Proof of Proposition \ref{2nd step}.}
Let $M\in \NN$, $\x_m \in \Gamma$ and $\eta_m \in G$  be as in Lemma \ref{2ndstepl1}.  Then
\begin{equation}\label{2nd step eq1}
ev(\hat{w}) \equiv ev(\hat{v})\prod_{m=1}^M [ \widehat{\x_m}, \widehat{\eta_m}] \quad    (\!\!\!\! \mod [F,R]) \, .
\end{equation}
Since $\x_1\in \Gamma$, there exists $k \in \{1, \dots , d\}$ such that $v_k$ is conjugate to $\x_1$. 
Without loss of generality, assume $k=d$. 
Arguing as in  Proposition \ref{1st step}, there exists $\varphi'_1\in Map(g'+1,d)$  such that
\begin{eqnarray}\label{2nd step eq 2}
&& ev(\widehat{\varphi'_1 \cdot w^1}) \equiv ev(\widehat{\varphi'_1 \cdot v^1})\prod_{m=1}^M [ \widehat{\x_m}, \widehat{\eta_m}] \quad    (\!\!\!\! \mod [F,R]) \, , \\
&&(\varphi'_1 \cdot w^1)_i=(\varphi'_1 \cdot v^1)_i \, , \quad \forall i=1, \dots , d \, ,  \nonumber \\
&& (\varphi'_1 \cdot w^1)_d=\x_1 \, . \nonumber
\end{eqnarray}
Equation  \eqref{2nd step eq 2} holds true because 
of \eqref{2nd step eq1}, Lemma \ref{2nd step l3} and Lemma \ref{zimmer+}. 
Moreover there exists also $\varphi_1\in Map(g'+2,d)$  such that
$\varphi_1 \cdot v^2 = (\varphi'_1 \cdot v^1)^1$, the vector obtained by adding one handle with 
trivial monodromies to $\varphi_1 \cdot v^1$.

By the argument of the proof of
Lemma \ref{zimmer+}, we have that this vector is 
$Map(g'+2,d)$-equivalent to
$$
 (v_1, \dots , v_{d-1}, \x_1 ; \eta_1 , 1, 
 (\varphi_1'\cdot v^1)_{d+1}, (\varphi_1'\cdot v^1)_{d+2},
 v_{d+1}, \dots , v_{d+2g'}) \, .
$$
using only transformations as in Lemma \ref{2nd step l3} i).
Here we use the assumption that the $v_j$, $j>d$, generate $G$.

Using the automorphisms
of \cite[2.1.a) and b)]{Zimmer}  the vector is further equivalent to
$$
 (v_1, \dots , v_{d-1}, \x_1 ; 1, \eta_1 , 
 (\varphi_1'\cdot v^1)_{d+1}, (\varphi_1'\cdot v^1)_{d+2},
 v_{d+1}, \dots , v_{d+2g'}) \, .
$$
Now, by \eqref{CLP} (Proposition 6.2 (i) with $\ell=1$ of \cite{CLP12})  there exists $\varphi_2 \in Map(g'+2,d)$ such that
\begin{eqnarray*}
&& \varphi_2 \cdot (\varphi_1 \cdot v^2)\\  
& = & 
(v_1, \dots , v_{d-1}, \x_1\eta_1\x_1\eta_1\inv\x_1\inv ; \x_1, \eta_1 , 
 (\varphi_1'\cdot v^1)_{d+1}, (\varphi_1'\cdot v^1)_{d+2},
 v_{d+1}, \dots , v_{d+2g'}) \, \\
& = & 
(v_1, \dots , v_{d-1}, \x_1 ; \x_1, \eta_1 , 
 (\varphi_1'\cdot v^1)_{d+1}, (\varphi_1'\cdot v^1)_{d+2},
 v_{d+1}, \dots , v_{d+2g'}) \, ,
\end{eqnarray*}
since $[\x_1, \eta_1]=1$ according to Lemma \ref{2ndstepl1}.
The same property implies
$ev(\widehat{\!\varphi_2 \! \cdot \!\varphi_1 \!\! \cdot \! v^2\!})= ev(\widehat{\varphi_1 \cdot v^2})[\widehat{\x_1}, \widehat{\eta_1}]$.

This proves the desired assertion  if $M=1$. The general case, where  $M>1$, is proven  inductively along the same lines.

\qed


\medskip
\noindent {\bf Authors' Addresses:}\\

\noindent Fabrizio Catanese,\\ 
 Lehrstuhl Mathematik VIII,  Mathematisches
Institut der Universit\"at
Bayreuth\\ NW II,  Universit\"atsstr. 30,  95447 Bayreuth (Germany).\\

Michael L\"onne, \\
Institut f\"ur Algebraische Geometrie, Gottfried Wilhelm Leibniz Universit\"at Hannover\\
Welfengarten 1, 30167 Hannover (Germany).\\

Fabio Perroni, \\
SISSA - International School for Advanced Studies, School of Mathematics\\
 via Bonomea, 265 - 34136 Trieste (Italy).\\

          email:                 
        fabrizio.catanese@uni-bayreuth.de,
loenne@math.uni-hannover.de, 
perroni@sissa.it

\end{document}